\documentclass[11pt, a4paper,leqno]{amsart}
\usepackage{amsmath,amsthm,amscd,amssymb,amsfonts, amsbsy}
\usepackage{latexsym}
\usepackage{txfonts}
\usepackage{exscale}

\usepackage[colorlinks=true, pdfstartview=FitV, linkcolor=blue, citecolor=red, urlcolor=blue]{hyperref}

\calclayout
\allowdisplaybreaks

\theoremstyle{plain}
\newtheorem{theorem}[equation]{Theorem}
\newtheorem{lemma}[equation]{Lemma}

\theoremstyle{definition}
\newtheorem{definition}[equation]{Definition}

\theoremstyle{remark}
\newtheorem{remark}[equation]{Remark}

\numberwithin{equation}{section}

\newcommand{\RR}{{\mathbb{R}}}
\newcommand{\eps}{\varepsilon}

\newcommand{\dint}{\int\!\!\!\int}

\newcommand{\dist}{\operatorname{dist}}

\newcommand{\tom}{\omega_{\star}}
\newcommand{\ts}{\sigma_{\star}}

\newcommand{\re}{\mathbb{R}}

\newcommand{\ree}{\mathbb{R}^{n+1}}

\newcommand{\dd}{\mathbb{D}}

\newcommand{\C}{\mathcal{C}}

\newcommand{\F}{\mathcal{F}}

\newcommand{\M}{\mathcal{M}}
\newcommand{\W}{\mathcal{W}}

\newcommand{\R}{\mathcal{R}}
\newcommand{\s}{\mathcal{S}}
\newcommand{\oo}{\mathcal{O}}

\newcommand{\Pf}{\mathcal{P}_{\mathcal{F}}}
\newcommand{\Pfo}{\mathcal{P}_{\mathcal{F}_0}}

\newcommand{\mut}{\mathfrak{m}}
\newcommand{\NN}{\mathfrak{N}}
\newcommand{\pom}{\partial\Omega}

\newcommand{\hm}{\omega}

\renewcommand{\P}{\mathcal{P}}

\renewcommand{\emptyset}{\mbox{\textup{\O}}}

\DeclareMathOperator{\diam}{diam}
\DeclareMathOperator{\interior}{int}

\begin{document}
\allowdisplaybreaks

\title[Uniform Rectifiability and Harmonic Measure III]{Uniform Rectifiability and Harmonic Measure III:  Riesz transform bounds imply uniform rectifiability of boundaries of 1-sided NTA domains}

\author{Steve Hofmann}

\address{Steve Hofmann
\\
Department of Mathematics
\\
University of Missouri
\\
Columbia, MO 65211, USA} \email{hofmanns@missouri.edu}

\author{Jos\'{e} Mar\'{\i}a Martell}

\address{Jos\'{e} Mar\'{\i}a Martell\\
Instituto de Ciencias Matem\'{a}ticas CSIC-UAM-UC3M-UCM\\
Consejo Superior de Investigaciones Cient\'{\i}ficas\\
C/ Nicol\'{a}s Cabrera, 13-15\\
E-28049 Madrid, Spain} \email{chema.martell@icmat.es}

\author{Svitlana Mayboroda}

\address{Svitlana Mayboroda
\\
Department of Mathematics
\\
University of Minnesota
\\
Minneapolis, MN 55455, USA} \email{svitlana@math.umn.edu}

\thanks{The first author was supported by the NSF. The second author was supported by MINECO Grant MTM2010-16518 and ICMAT Severo Ochoa project SEV-2011-0087. The third author was  supported by the Alfred P. Sloan Fellowship, the NSF CAREER Award DMS 1056004,  and the NSF Materials Research Science and Engineering Center Seed Grant.}

\date{July 3, 2012. \textit{Revised}: \today}



\begin{abstract}
Let $E\subset \ree$, $n\ge 2$, be a closed, Ahlfors-David regular set of dimension $n$ satisfying the ``Riesz Transform bound"
$$\sup_{\eps>0}\int_E\left|\int_{\{y\in E:|x-y|>\eps\}}\frac{x-y}{|x-y|^{n+1}} \,f(y)\, dH^n(y)\right|^2 dH^n(x) \,\leq \,C
\int_E|f|^2 dH^n\,.$$  Assume further that $E$ is the boundary of a domain $\Omega\subset\ree$ satisfying the Harnack Chain condition plus an
interior (but not exterior) Corkscrew condition. Then $E$ is uniformly rectifiable.
\end{abstract}

\maketitle

\tableofcontents

\section{Introduction}

In 1991 David and Semmes proved that for a given Ahlfors-David regular set $E$ the boundedness of {\it all} singular integral operators  with odd infinitely smooth kernels in $L^2(E)$ is equivalent to the property  that $E$ is Uniformly Rectifiable \cite{DS1}. Rectifiability means that $E$ consists of a countable number of images of Lipschitz mappings, modulo a set of measure zero, and Uniform Rectifiability is a certain quantitative, scale-invariant, version of this property.
While the method of proof in \cite{DS1} required use of a large class of singular integral operators,
it has been conjectured by the authors that  the boundedness in $L^2(E)$ of {\it one} key singular operator, the Riesz transform, is sufficient for Uniform Rectifiability of $E$.

The  two-dimensional version of the Riesz transform conjecture, i.e.,
the fact that boundedness of the Cauchy transform in $L^2(E)$ (where $E$ is a 1-dimensional ADR set in the plane)
implies Uniform Rectifiability of $E$, has been established by Mattila, Melnikov and Verdera  \cite{MMV}, using the so-called Menger curvature.
It was formally shown, however, that the Menger curvature approach is rigidly restricted to
$n+1=2$ \cite{F}.
This result, and the development of related ideas, had significant consequences,
including the eventual resolution, in the remarkable work of X. Tolsa \cite{T}, of a conjecture of Vitushkin 
concerning the semi-additivity of analytic capacity, and the related 1880 Painlev\'{e} problem of characterizing removable singularities of analytic functions in metric/geometric terms;  see also
the earlier related work of Guy David \cite{DavidPain}, as well as the monograph of A. Volberg \cite{Vo}, where the author proves a higher dimensional analogue of Vitushkin's conjecture.

The main result of the present paper is as follows.
\begin{theorem}\label{theor:main-RT}
Let $E\subset \ree$, $n\ge 2$,  be a closed, $n$-dimensional Ahlfors-David regular set such that the Riesz transform is bounded in $L^2(E)$, that is,
\begin{equation}\label{main-RT}
\sup_{\eps>0}\int_E\left|\int_{\{y\in E:|x-y|>\eps\}}\frac{x-y}{|x-y|^{n+1}} \,f(y)\, dH^n(y)\right|^2 dH^n(x) \,\leq \,C_E
\int_E|f|^2 dH^n\,.
\end{equation}
Assume further that $E$ is the boundary of a connected open set
$\Omega\subset \ree$ that satisfies interior (but not necessarily
exterior) Corkscrew and Harnack Chain conditions. Then $E$ is uniformly rectifiable.
\end{theorem}

The precise definitions of the Ahlfors-David regularity, Corkscrew, and Harnack chain conditions are somewhat technical and will be given below (cf. Definitions~\ref{def1.cork}, \ref{def1.hc}, \ref{def1.ADR}). In rough terms, the Ahlfors-David regularity is  a natural ambient assumption, essentially saying that the set $E$ is $n$-dimensional at all scales. The interior Corkscrew and Harnack Chain conditions are scale invariant analogues of the topological properties of openness and path connectedness, respectively.
We shall use the terminology that a connected open set
$\Omega\subset \ree$ is a ``1-sided NTA domain" if it satisfies interior (but not necessarily
exterior) Corkscrew and Harnack Chain conditions.  We remark that, in the presence of
Ahlfors-David regularity, by the result of \cite{DJe},
a ``2-sided" Corkscrew condition (i.e., the existence of a Corkscrew point in two different components of
$\ree\setminus E$ at every scale) would already imply uniform rectifiability, without further conditions.
On the other hand, if $F$ denotes the ``four corners Cantor set" of Garnett, then $F$ is Ahlfors-David regular, and is the boundary of the 1-sided NTA domain $\mathbb{R}^2\setminus F$, but is
totally non-rectifiable.
Moreover, the Ahlfors-David regularity condition alone implies that some component of
$\ree\setminus E$ contains a Corkscrew point, at every scale.
Thus, from the point of view of rectifiability, there is a significant difference between the ``1-sided" and the ``2-sided" cases.

We note that, during the
preparation of this manuscript, we learned that F. Nazarov, X. Tolsa, and A. Volberg
\cite{NToV} have recently obtained a solution to the full conjecture, that is,
they prove Theorem~\ref{theor:main-RT} without the 1-sided NTA assumption. The manuscript of their work is expected to appear soon.  Although the present paper treats the conjecture only in a special case, we believe that our method of proof, based on harmonic measure techniques, may be of independent interest, and may eventually provide an alternative approach to the full conjecture.

As mentioned above, our approach to the proof of Theorem~\ref{theor:main-RT} is different from that of the aforementioned works in the subject, including \cite{NToV}.
We exploit in an essential way harmonicity of the Riesz transform, as well as recent results in \cite{HM-I}, \cite{HMU} which relate properties of harmonic measure to Uniform Rectifiability. In \cite{HM-I},
and jointly with I. Uriarte-Tuero in \cite{HMU}, the first two authors of the present paper have shown that for an Ahlfors-David regular set $E$, which is the boundary of a 1-sided NTA domain, Uniform Rectifiability is equivalent to the
(weak) $A_\infty$ property of harmonic measure.  Both the method of the proof in \cite{HM-I}, \cite{HMU}, involving Extrapolation of Carleson Measures techniques, and the result itself have been employed in the proof of Theorem~\ref{theor:main-RT}.

We mention, without discussing in detail, an extensive list of related results in the subject, including, but not limited to, connections between existence of principal values for Riesz transforms and rectifiability \cite{PVandRect}, different versions of square functions and rectifiability \cite{DS2}, \cite{MV2}, \cite{TolsaUnif}
and a rich array of geometric and analytic characterizations of Uniform Rectifiability in \cite{DS1}, \cite{DS2}.

We pass at this point to a rigorous list of Definitions of the main concepts.

\subsection{Notation and Definitions}

We borrow some notation and definitions from \cite{HM-I} and \cite{HMU}. For the sake of conciseness we will omit some of the background material developed there and the reader is recommended to have both references handy.

\begin{list}{$\bullet$}{\leftmargin=0.4cm  \itemsep=0.2cm}

\item Harmless positive constants will be denoted by $c,C$ and may change at each occurrence. Unless otherwise specified these will depend only on dimension and the ``allowable parameters'',  i.e., the
constants appearing in the hypotheses of Theorem \ref{theor:main-RT}.
We shall also sometimes write $a\lesssim b$ and $a \approx b$ to mean, respectively,
that $a \leq C b$ and $0< c \leq a/b\leq C$, where the constants $c$ and $C$ are as above.

\item Given a domain $\Omega \subset \ree$, we shall
use lower case letters $x,y,z$, etc., to denote points on $\partial \Omega$, and capital letters
$X,Y,Z$, etc., to denote generic points in $\ree$ (especially those in $\ree\setminus \partial\Omega$). We caution the reader that, when working with a subdomain $\Omega'$ of $\Omega$, we shall use
lower case letters to denote points on the boundary of $\Omega'$, even though these may be
interior points of $\Omega$.

\item The open $(n+1)$-dimensional Euclidean ball of radius $r$ will be denoted
$B(x,r)$ when the center $x$ lies on $\partial \Omega$, or $B(X,r)$ when the center
$X \in \ree\setminus \partial\Omega$.  A ``surface ball'' is denoted
$\Delta(x,r):= B(x,r) \cap\partial\Omega.$

\item Given a Euclidean ball $B$, or surface ball $\Delta$, its radius will be denoted
$r_B$ or $r_\Delta$, respectively.

\item Given a Euclidean or surface ball $B= B(X,r)$, or $\Delta = \Delta(x,r)$, its concentric
dilate by a factor of $\kappa >0$ will be denoted
by $\kappa B := B(X,\kappa r)$ or $\kappa \Delta := \Delta(x,\kappa r).$

\item For $X \in \ree$, we set $\delta(X):= \dist(X,\partial\Omega)$.

\item We let $H^n$ denote $n$-dimensional Hausdorff measure, and let
$\sigma := H^n\big|_{\partial\Omega}$ denote the ``surface measure'' on $\partial \Omega$.

\item For a Borel set $A\subset \ree$, we let $1_A$ denote the usual
indicator function of $A$, i.e. $1_A(x) = 1$ if $x\in A$, and $1_A(x)= 0$ if $x\notin A$.

\item For a Borel set $A\subset \ree$,  we let $\interior(A)$ denote the interior of $A$.
If $A\subset \partial\Omega$, then $\interior(A)$ will denote the relative interior, i.e., the largest relatively open set in $\partial\Omega$ contained in $A$.  Thus, for $A\subset \partial\Omega$,
the boundary is then well defined by $\partial A := \overline{A} \setminus {\rm int}(A)$.

\item We shall use the letter $I$ (and sometimes $J$)
to denote a closed $(n+1)$-dimensional Euclidean cube with sides
parallel to the co-ordinate axes, and we let $\ell(I)$ denote the side length of $I$.
We use $Q$ to denote a dyadic ``cube''
on $\partial \Omega$.  The
latter exist, given that $\partial \Omega$ is ADR  (cf. \cite{DS1}, \cite{Ch}), and enjoy certain properties
which we enumerate in Lemma \ref{lemmaCh} below.

\end{list}

\begin{definition} ({\bf Corkscrew condition}).  \label{def1.cork}
Following
\cite{JK}, we say that a domain $\Omega\subset \ree$
satisfies the ``Corkscrew condition'' if for some uniform constant $c>0$ and
for every surface ball $\Delta:=\Delta(x,r),$ with $x\in \partial\Omega$ and
$0<r<\diam(\partial\Omega)$, there is a ball
$B(X_\Delta,cr)\subset B(x,r)\cap\Omega$.  The point $X_\Delta\subset \Omega$ is called
a ``Corkscrew point'' relative to $\Delta.$  We note that  we may allow
$r<C\diam(\pom)$ for any fixed $C$, simply by adjusting the constant $c$.
\end{definition}

\begin{definition}({\bf Harnack Chain condition}).  \label{def1.hc} Again following \cite{JK}, we say
that $\Omega$ satisfies the Harnack Chain condition if there is a uniform constant $C$ such that
for every $\rho >0,\, \Lambda\geq 1$, and every pair of points
$X,X' \in \Omega$ with $\delta(X),\,\delta(X') \geq\rho$ and $|X-X'|<\Lambda\,\rho$, there is a chain of
open balls
$B_1,...,B_N \subset \Omega$, $N\leq C(\Lambda)$,
with $X\in B_1,\, X'\in B_N,$ $B_k\cap B_{k+1}\neq \emptyset$
and $C^{-1}\diam (B_k) \leq \dist (B_k,\partial\Omega)\leq C\diam (B_k).$  The chain of balls is called
a ``Harnack Chain''.
\end{definition}



\begin{definition}\label{def1.ADR}
({\bf Ahlfors-David regular}). We say that a closed set $E \subset \ree$ is $n$-dimensional ADR (or simply ADR) if
there is some uniform constant $C$ such that
\begin{equation} \label{eq1.ADR}
\frac1C\, r^n \leq H^n(E\cap B(x,r)) \leq C\, r^n,\,\,\,\forall r\in(0,R_0),x \in E,\end{equation}
where $R_0$ is the diameter of $E$ (which may be infinite).   When $E=\partial \Omega$,
the boundary of a domain $\Omega$, we shall sometimes for convenience simply
say that ``$\Omega$ has the ADR property'' to mean that $\partial \Omega$ is ADR.
\end{definition}

\begin{definition}\label{def1.UR}
({\bf Uniform Rectifiability}).
We say
that a closed set $E\subset \ree$ is    $n$-dimensional UR (or simply UR) (``Uniformly Rectifiable''), if
it satisfies the ADR condition \eqref{eq1.ADR}, and if for some uniform constant $C$ and for every
Euclidean ball $B:=B(x_0,r), \, r\leq \diam(E),$ centered at any point $x_0 \in E$, we have the Carleson measure estimate
\begin{equation}\label{eq1.sf}\dint_{B}
|\nabla^2 \mathcal{S}1(X)|^2 \,\dist(X,E) \,dX \leq C r^n,
\end{equation}
where $\mathcal{S}f$ is the single layer potential of $f$, i.e.,
\begin{equation}\label{eq1.layer}
\mathcal{S}f(X) :=c_n\, \int_{E} |X-y|^{1-n} f(y) \,dH^n(y).
\end{equation}
Here, the normalizing constant $c_n$ is chosen so that
$\mathcal{E}(X) := c_n |X|^{1-n}$ is the usual fundamental solution for the Laplacian
in $\ree.$  When $E=\partial \Omega$,
the boundary of a domain $\Omega$, we shall sometimes for convenience simply
say that ``$\Omega$ has the UR property'' to mean that $\partial \Omega$ is UR.
\end{definition}

UR sets may be characterized in many ways, and were originally defined
in more explicitly geometrical terms (see \cite{DS1,DS2});  Definition \ref{def1.UR} is obtained as
a characterization of such sets in \cite{DS2}.
As mentioned above, the UR sets
are precisely those for which all ``sufficiently nice'' singular integrals are bounded on $L^2$ (see \cite{DS1}).

\noindent {\it Acknowledgements}.
We thank Ignacio Uriarte-Tuero for bringing the work of Nazarov, Tolsa and Volberg to our attention.  The first named author also thanks Xavier Tolsa for explaining
the proof of the main result of \cite{NToV}.  We are also most grateful to the referee for a careful reading of the manuscript, and for numerous suggestions to improve the exposition.

\section{Proof of Theorem \ref{theor:main-RT}}

\subsection{Preliminaries}

We collect some of the definitions and auxiliary results from \cite{HM-I} that will be used later. The reader is referred to \cite{HM-I} for more details.

\subsubsection{Dyadic Grids and Sawtooth domains}

\begin{lemma}\label{lemmaCh}({\bf Existence and properties of the ``dyadic grid''})
\cite{DS1,DS2}, \cite{Ch} (also \cite{HM-I} for specific notation.)
Suppose that $E\subset \ree$ satisfies
the ADR condition \eqref{eq1.ADR}.  Then
there is a collection of Borel sets (``cubes'')
$$
\dd=\dd(E)=\cup_k \dd_k,
$$
satisfying

\begin{list}{$(\theenumi)$}{\usecounter{enumi}\leftmargin=.8cm
\labelwidth=.8cm\itemsep=0.2cm\topsep=.1cm
\renewcommand{\theenumi}{\roman{enumi}}}

\item $E=\cup_{Q\in\dd_k}Q$ for each
$k\in{\mathbb Z}$ and this union is comprised of disjoint sets.

\item If $Q$, $Q'\in\dd$ and $Q\cap Q'\neq\emptyset$ then either $Q\subset Q'$ or $Q'\subset Q$.

\item For each $Q\in\dd_k$ and each $m<k$, there is a unique
$Q'\in\dd_m$ with $Q\subset Q'$.

\item For each cube $Q\in\dd_k$, set $\ell(Q) = 2^{-k}$. We shall refer to this quantity as the ``length''
of $Q$.  Then,  $\ell(Q)\approx \diam(Q)$ and there is a point $x_Q\in E$, referred to as the ``center'' of $Q$,
a Euclidean ball $B(x_Q,r)$ and a surface ball
$\Delta(x_Q,r):= B(x_Q,r)\cap E$ such that
$r\approx 2^{-k} \approx {\rm diam}(Q)$
and \begin{equation}\label{cube-ball}
\Delta(x_Q,r)\subset Q \subset \Delta(x_Q,Cr),\end{equation}
for some uniform constant $C$.
We shall denote the respective balls by
\begin{equation}\label{cube-ball2}
B_Q:= B(x_Q,r) \,,\qquad\Delta_Q:= \Delta(x_Q,r).
\end{equation}
\end{list}
\end{lemma}

\bigskip

A few remarks are in order concerning this lemma.

\begin{list}{$\bullet$}{\leftmargin=0.4cm  \itemsep=0.2cm}

\item In the setting of a general space of homogeneous type, this lemma has been proved by Christ
\cite{Ch}.  In that setting, the
dyadic parameter $1/2$ should be replaced by some constant $\delta \in (0,1)$.
It is a routine matter to verify that one may take $\delta = 1/2$ in the presence of the Ahlfors-David
property (\ref{eq1.ADR}) (in this more restrictive context, the result already appears in \cite{DS1,DS2}).

\item  For our purposes, we may ignore those
$k\in \mathbb{Z}$ such that $2^{-k} \gtrsim {\rm diam}(E)$, in the case that the latter is finite.

\item Let us now specialize to the case that  $E=\pom$, with $\Omega$ satisfying the Corkscrew condition.  Given $Q\in \mathbb{D}(\partial\Omega)$,
we
shall sometimes refer to a ``Corkscrew point relative to $Q$'', which we denote by
$X_Q$, and which we define to be the corkscrew point $X_\Delta$ relative to the ball
$\Delta:=\Delta_Q$ (cf. \eqref{cube-ball}, \eqref{cube-ball2} and Definition \ref{def1.cork}).  We note that
\begin{equation}\label{eq1.cork}
\delta(X_Q) \approx \dist(X_Q,Q) \approx \diam(Q).
\end{equation}

\item For a dyadic cube $Q \in \mathbb{D}$, we let $k(Q)$ denote the ``dyadic generation''
to which $Q$ belongs, i.e., we set  $k = k(Q)$ if
$Q\in \mathbb{D}_k$; thus, $\ell(Q) =2^{-k(Q)}$.

\end{list}

\bigskip

We next introduce some ``discretized'' and ``geometric'' sawtooth and Carleson regions from \cite[Section 3]{HM-I}.
Given a ``dyadic cube'' $Q\in \dd(\partial\Omega)$, the {\bf discretized Carleson region} $\dd_Q$ is defined to be
\begin{equation}\label{eq2.discretecarl}
\dd_Q:= \left\{Q'\in \dd: Q'\subset Q\right\}.
\end{equation}
Given a family $\mathcal{F}$ of disjoint cubes $\{Q_j\}\subset \mathbb{D}$, we define
the {\bf global discretized sawtooth} relative to $\F$ by
\begin{equation}\label{eq2.discretesawtooth1}
\dd_{\F}:=\dd\setminus \bigcup_{\F} \dd_{Q_j}\,,
\end{equation}
i.e., $\dd_{\F}$ is the collection of all $Q\in\dd$ that are not contained in any $Q_j\in\F$.
Given some fixed cube $Q$,
the {\bf local discretized sawtooth} relative to $\F$ by
\begin{equation}\label{eq2.discretesawtooth2}
\dd_{\F,Q}:=\dd_Q\setminus \bigcup_{\F} \dd_{Q_j}=\dd_\F\cap\dd_Q.
\end{equation}

We also introduce the ``geometric'' Carleson regions and sawtooths. In the sequel, $\Omega \subset \ree,\, n\geq 2,$ will be a 1-sided NTA domain with ADR boundary. Let $\mathcal{W}=\W(\Omega)$ denote a collection
of (closed) dyadic Whitney cubes of $\Omega$, so that the cubes in $\mathcal{W}$
form a covering of $\Omega$ with non-overlapping interiors,  which satisfy
\begin{equation}\label{eqWh1} 4\, {\rm{diam}}\,(I)\leq \dist(4 I,\pom) \leq  \dist(I,\pom) \leq 40 \, {\rm{diam}}\,(I)\end{equation}
and also
\begin{equation}\label{eqWh2}\diam(I_1)\approx \diam(I_2), \mbox{ whenever $I_1$ and $I_2$ touch.}\end{equation}
Let $X(I)$ denote the center of $I$ and let $\ell(I)$ denote the side length of $I$,
and write $k=k_I$ if $\ell(I) = 2^{-k}$.

Given $0<\lambda<1$ and $I\in\W$ we write $I^*=(1+\lambda)I$ for the ``fattening'' of $I$. By taking $\lambda$ small enough,  we can arrange matters so that, first, $\dist(I^*,J^*) \approx \dist(I,J)$ for every
$I,J\in\W$, and secondly, $I^*$ meets $J^*$ if and only if $\partial I$ meets $\partial J$
(the fattening thus ensures overlap of $I^*$ and $J^*$ for
any pair $I,J \in\W$ whose boundaries touch, so that the
Harnack Chain property then holds locally, with constants depending upon $\lambda$, in $I^*\cup J^*$).  By choosing $\lambda$ sufficiently small,
we may also suppose that there is a $\tau\in(1/2,1)$ such that for distinct $I,J\in\W$, $\tau J\cap I^* = \emptyset$.

For every $Q$ we can construct a family $\W_Q^*\subset \W$ and define
\begin{equation}\label{eq2.whitney3}
U_Q := \bigcup_{I\in\,\mathcal{W}^*_Q} I^*\,,
\end{equation}
satisfying the following properties:
$X_Q\in U_Q$ and there are uniform constants $k^*$ and $K_0$ such that
\begin{eqnarray}\label{eq2.whitney2}
& k(Q)-k^*\leq k_I \leq k(Q) + k^*\,, \quad \forall I\in \mathcal{W}^*_Q\\\nonumber
&X(I) \rightarrow_{U_Q} X_Q\,,\quad\forall I\in \mathcal{W}^*_Q\\ \nonumber
&\dist(I,Q)\leq K_0\,2^{-k(Q)}\,, \quad \forall I\in \mathcal{W}^*_Q\,.
\end{eqnarray}
Here $X(I) \rightarrow_{U_Q} X_Q$ means that the interior of $U_Q$ contains all the balls in
a Harnack Chain (in $\Omega$), connecting $X(I)$ to $X_Q$, and  moreover, for any point $Z$ contained
in any ball in the Harnack Chain, we have $
\dist(Z,\pom) \approx \dist(Z,\Omega\setminus U_Q)\,,
$
with uniform control of the implicit constants.
The constants  $k^*$, $K_0$ and the implicit constants in the condition $X(I)\to_{U_Q} X_Q$ in \eqref{eq2.whitney2}
depend only  on the ``allowable
parameters'' and on $\lambda$. The reader is refereed to \cite{HM-I} for full details.

We may then define  the {\bf Carleson box}
associated to $Q$ by
\begin{equation}\label{eq2.box}
T_Q:={\rm int}\left( \bigcup_{Q'\in \dd_Q} U_{Q'}\right).
\end{equation}
Similarly, we may define geometric sawtooth regions as follows.
As above, given a family $\mathcal{F}$ of disjoint cubes $\{Q_j\}\subset \mathbb{D}$,
we define the {\bf global sawtooth} relative to $\mathcal{F}$ and the {\bf local sawtooth} relative to $\mathcal{F}$
for some fixed $Q\in \dd$ by
\begin{equation}\label{eq2.sawtooth1}
\Omega_{\mathcal{F}}:= {\rm int } \left( \bigcup_{Q'\in\dd_\F} U_{Q'}\right)\,,
\qquad
\Omega_{\mathcal{F},Q}:=  {\rm int } \left( \bigcup_{Q'\in\dd_{\F,Q}} U_{Q'}\right)\,.
\end{equation}

\medskip

Analogously we can define fattened versions of the aforementioned regions as follows:
\begin{equation}\label{UQ-fat} U_Q^{fat} := \bigcup_{I\in\,\mathcal{W}^*_Q} 4I,\quad
\Omega^{fat}_{\mathcal{F}}:=  {\rm int } \left( \bigcup_{Q'\in\dd_{\F}} U^{fat}_{Q'}\right),
\quad
T^{fat}_{Q}:=  {\rm int } \left( \bigcup_{Q'\in\dd_Q} U^{fat}_{Q'}\right),
\end{equation}
and, similarly,
\begin{equation}\label{UQ-fat*} U_Q^{fat*} := \bigcup_{I\in\,\mathcal{W}^*_Q} 5I,\quad
\Omega^{fat*}_{\mathcal{F}}:=  {\rm int } \left( \bigcup_{Q'\in\dd_{\F}} U^{fat*}_{Q'}\right).\end{equation}

Let us also introduce a non-tangential maximal function: for every $x\in\pom$ we define
\begin{equation}\label{eqNTdef}
N_* u(x)=\sup_{X\in\Gamma(x)}|u(X)|,
\qquad
\Gamma(x)=\bigcup_{x\in Q\in\dd} U^{fat*}_{Q}.
\end{equation}
Note that the sets $U^{fat*}_{Q}$ are slightly fatter than $U^{fat}_{Q}$, which, in a sense, enlarges aperture of underlying cones.

A domain $\Omega$ is said to satisfy the \textbf{qualitative exterior Corkscrew} condition if there exists $N\gg 1$ such that $\Omega$ has exterior corkscrew points at all scales smaller than $2^{-N}$. That is, there exists a constant $c_N$ such that for every surface ball $\Delta=\Delta(x,r)$, with $x\in \pom$  and $r\le 2^{-N}$, there is a ball $B(X_\Delta^{ext},c_N\,r)\subset B(x,r)\cap\Omega_{ext}$,  where $\Omega_{ext}:= \ree\setminus \overline{\Omega}$.
Let us observe that if
$\Omega$ satisfies the qualitative exterior
Corkscrew condition, then every point in $\pom$ is regular in the sense of Wiener.
Moreover, for $1$-sided NTA domains, the qualitative
exterior Corkscrew points allow local H\"{o}lder continuity
at the boundary (albeit with bounds which may depend badly on $N$).
In the sequel, we shall work with certain approximating domains, for which
the qualitative exterior Corkscrew condition holds (cf. subsections \ref{sub2.2}-\ref{sub2.3} below).  Of course, none of our estimates will depend
quantitatively on this condition.

\medskip

\begin{lemma}[{\cite[Lemmata 3.61 and 3.62]{HM-I}}]\label{lemma2.30}  Suppose that $\Omega$ is a  1-sided NTA domain
with an ADR boundary.
Then all of its  Carleson boxes
$T_Q$, and sawtooth regions
$ \Omega_{\mathcal{F}}$, $\Omega_{\mathcal{F},Q}$ are also 1-sided NTA domains
with ADR boundaries.
In all cases, the implicit constants are uniform, and
depend only on dimension and on the corresponding constants for $\Omega$.

Furthermore, if $\Omega$ also satisfies the qualitative exterior Corkscrew condition, then all of its  Carleson boxes
$T_Q$ and sawtooth regions $ \Omega_{\mathcal{F}}$, $\Omega_{\mathcal{F},Q}$ satisfy the qualitative exterior
Corkscrew condition.
\end{lemma}

\subsubsection{Harmonic measure and Green function}

In the sequel, $\Omega \subset \ree,\, n\geq 2,$ will be a connected, open set, $\omega^X$ will denote
harmonic measure for $\Omega$, with pole at $X$. At least in the case that $\Omega$ is bounded, we may, as usual,
define $\omega^X$ via the maximum principle and the Riesz representation theorem,
after first using the method of Perron (see, e.g., \cite[pp. 24--25]{GT}) to construct a harmonic function ``associated'' to arbitrary continuous boundary
data.\footnote{Since we have made no assumption as regards
Wiener's regularity criterion, our harmonic function is a generalized solution, which
may not be continuous up to the boundary.}  For unbounded $\Omega$, we may still
define harmonic measure via a standard approximation scheme, see, e.g.,
\cite[Section 3]{HM-I} for more details.  We note for future reference that $\omega^X$
is a non-negative, finite, outer regular Borel measure.

By a result proved by Bourgain \cite{B} if $\partial \Omega$ is $n$-dimensional ADR, then there exist uniform constants $c\in(0,1)$
and $C\in (1,\infty)$,
such that for every $x \in \partial\Omega$, and every $r\in (0,\diam(\partial\Omega))$,
\begin{equation}\label{eq2.Bourgain1}
\omega^{Y} (\Delta(x,r)) \geq 1/C>0,
\qquad
Y \in \Omega \cap B(x,cr)
 \;.
\end{equation}
In particular, if $\Omega$ is a 1-sided NTA domain
then for every surface ball $\Delta$, we have
\begin{equation}\label{eq2.Bourgain2}
\omega^{X_\Delta} (\Delta) \geq 1/C>0 \;.
\end{equation}

The Green function is constructed by setting
\begin{equation}\label{eq2.greendef}
G(X,Y):= \mathcal{E}\,(X-Y) - \int_{\partial\Omega}\mathcal{E}\,(X-z)\,d\omega^Y(z),
\end{equation}
where $\mathcal{E}\,(X):= c_n |X|^{1-n}$ is the usual fundamental solution for the Laplacian
in $\ree$.  We choose the normalization that makes $\mathcal{E}$ positive. In such a way, the Green function satisfies the following standard
properties (see, e.g.,  \cite[Lemma 3.11]{HM-I}:

\begin{list}{$\bullet$}{\leftmargin=0.6cm  \itemsep=0.2cm}

\item $G(X,Y) \leq C\,|X-Y|^{1-n}$.

\item $G(X,Y)\ge c(n,\theta)\,|X-Y|^{1-n}$, if $|X-Y|\leq \theta\, \delta(X)$, $\theta \in (0,1)$.

\item If every point on $\pom$  is regular in the sense of Wiener, then $G(X,Y)\geq 0$ and $G(X,Y)=G(Y,X)$ for all $X,Y\in\Omega$, $X\neq Y$.

\end{list}

On the other hand, if $\Omega$ is a 1-sided NTA domain with $n$-dimensional ADR boundary such that $\Omega$ satisfies the qualitative exterior corkscrew condition we have the following
Caffarelli-Fabes-Mortola-Salsa estimates (see \cite[Lemma 3.30]{HM-I}):
Given $B_0:=B(x_0,r_0)$  with $x_0\in\pom$, and $\Delta_0 := B_0\cap\partial\Omega$, let $B:=B(x,r)$,  $x\in \pom$, and
$\Delta:=B\cap\pom$, and suppose that $2B\subset B_0.$  Then
\begin{equation}\label{CFMS}
\frac1C \frac{\omega^X(\Delta)}{\sigma(\Delta)}
\le
\frac{G(X_\Delta,X)}{r} \leq C \frac{\omega^X(\Delta)}{\sigma(\Delta)},
\qquad X\in\Omega\setminus B_0.
\end{equation}
The constant $C$ depends {\bf only} on dimension and on the constants in the $ADR$ and 1-sided NTA
conditions.  As a consequence of \eqref{CFMS} we obtain as usual that harmonic measure is doubling (see \cite[Corollary 3.36]{HM-I}). More precisely,
\begin{equation}\label{doubling}
\omega^X(2\Delta)\leq C\omega^X(\Delta),
\qquad X\in \Omega\setminus 4B,
\end{equation}
for every surface ball $\Delta=B\cap\pom$, where as before $C$ depends {\bf only} on dimension and on the constants in the $ADR$ and 1-sided NTA
conditions.

\begin{remark}
Let us emphasize that although we have assumed that $\Omega$ satisfies the qualitative exterior corkscrew condition, the constants in \eqref{CFMS} and \eqref{doubling} do not depend on that qualitative assumption (i.e., constants do not depend on $N$). This is because the argument in \cite{HM-I}  to obtain the left hand side inequality of \eqref{CFMS} uses this qualitative assumption (with parameter $N$) to know that it does \textit{a priori}  hold with a finite constant which may depend very badly on the parameter $N$. The fact that the constant is finite allows then to run a hiding argument to eventually show the desired estimate holds with a uniform bound. We refer the reader to \cite{HM-I} for full details.

\end{remark}

\begin{remark}\label{remark:inherit}
Given a 1-sided NTA domain $\Omega$ with ADR boundary satisfying the qualitative exterior corkscrew condition, we know by Lemma \ref{lemma2.30}, that all these properties are inherited by any
Carleson box $T_Q$, global sawtooth $ \Omega_{\mathcal{F}}$ or local sawtooth $\Omega_{\mathcal{F},Q}$. Furthermore, the implicit constants depend only on dimension and the original constants corresponding to $\Omega$. Thus, \eqref{CFMS} and \eqref{doubling} hold for the harmonic measures and Green functions associated to each of the previous domains with uniform bounds.
\end{remark}

\subsection{Scheme of the proof}\label{sub2.2}

Let $\Omega$ be a 1-sided NTA domain with ADR boundary such  that \eqref{main-RT} holds with $E=\pom$. As observed above, having in addition the qualitative exterior corkscrew condition gives us a much richer and friendlier environment. To explore this, we shall first transfer the Riesz transforms bounds from $\Omega$ to some approximating domains $\Omega_N$, which are $1$-sided NTA domains further satisfying the qualitative exterior corkscrew condition. In doing that we will be able to show that all bounds are uniform on $N$. This is Step 1 in the proof and the argument is essentially contained in \cite[Appendix C]{HM-I}. Once we are in this nicer setting, we will establish in Step 2 that boundedness of the Riesz transform implies uniform rectifiability for each $\Omega_N$, with the UR constants uniform in $N$. That is the heart of the proof and our main contribution. Finally, in Step 3 we will transfer the UR property from $\Omega_N$ to $\Omega$ using an argument in \cite{HMU}.

\subsection{Step 1: Passing to the approximating domains}\label{sub2.3}

We define approximating domains as follows.
For each large integer $N$,  set $\F_N := \dd_N$.  We then let
$\Omega_N := \Omega_{\F_N}$ denote the usual (global) sawtooth with respect to the family
$\F_N$ (cf. \eqref{eq2.whitney2}, \eqref{eq2.whitney3} and \eqref{eq2.sawtooth1}.)  Thus,
\begin{equation}\label{eq7.on}
\Omega_N =\interior\left(\bigcup_{Q\in \dd:\,\ell(Q)\geq 2^{-N+1}}U_Q\right),
\end{equation}
so that $\overline{\Omega_N}$ is the union of fattened Whitney cubes $I^*=(1+\lambda)I$, with $\ell(I)\gtrsim 2^{-N}$, and the boundary of $\Omega_N$ consists of portions of faces of $I^*$ with $\ell(I)\approx 2^{-N}$.
By virtue of Lemma \ref{lemma2.30}, each $\Omega_N$ satisfies the ADR,
Corkscrew and Harnack Chain properties.  We note that, for each of these properties, the constants are uniform in $N$, and depend only on dimension and on the corresponding constants for $\Omega$.

By construction $\Omega_N$ satisfies the qualitative exterior corkscrew condition since it has exterior corkscrew points at
all scales $\lesssim 2^{-N}$.  By Lemma \ref{lemma2.30}  the same statement applies to the Carleson boxes $T_Q$, and to the sawtooth domains $\Omega_\F$ and $\Omega_{\F,Q}$  (all of them relative to $\Omega_N$)  and even to Carleson boxes within sawtooths.

We note that \eqref{main-RT} implies that there exists a constant $C_E'$ depending {\bf only} on dimension, the constants in the $ADR$ and 1-sided NTA
conditions of $\Omega$ and the constant $C_E$ in \eqref{main-RT} such that for all $N\gg 1$ we have
\begin{equation}\label{main-RT:OmegaN}
\sup_{\eps>0}\int_{\pom_N}\left|\int_{\{y\in\pom_N:|x-y|>\eps\}}\frac{x-y}{|x-y|^{n+1}} \,f(y)\, dH^n(y)\right|^2 dH^n(x) \,
\leq \,C_E'
\int_{\pom_N}|f|^2 dH^n\,.
\end{equation}
This fact is proved in \cite[Appendix C]{HM-I}, using ideas of Guy David.
To avoid possible confusion, we point out that the result in \cite[Appendix C]{HM-I} states that the UR property may be inferred, uniformly, for the subdomains $\Omega_N$, given that it holds for $\Omega$, but the proof actually shows that the $L^2$ boundedness of any given singular integral operator, on $\pom$,
may be transferred uniformly to $\pom_N$.  Moreover, in \cite[Appendix C]{HM-I}, one works with smoothly truncated singular integrals, but the error between these and the sharp truncations considered in \eqref{main-RT:OmegaN} is controlled by the Hardy-Littlewood maximal function.

\subsection{Step 2: The heart of the proof}

Having established Step 1, we are in a position  to show that \eqref{main-RT:OmegaN} implies uniform rectifiability of $\Omega_N$ with bounds uniform in $N$. In the last step we shall prove that this ultimately implies that $\Omega$ inherits the UR property.

Let us recall that, by virtue of Lemma \ref{lemma2.30}, $\Omega_N$ is a 1-sided NTA domain with ADR boundary and all the NTA and ADR constants are independent of $N$. Moreover, $\Omega_N$ satisfies the qualitative exterior corkscrew condition and thus, \eqref{CFMS}, \eqref{doubling} hold with constants which are, once again, independent of $N$. By assumptions and Step 1,
we have also \eqref{main-RT:OmegaN}, that is,  the boundedness of Riesz transforms (in $\pom_N$) with bounds that do not depend on $N$.

Therefore, abusing the notation, in this step we can drop the index 
$N$ everywhere and write $\Omega$ to denote the corresponding approximating domain $\Omega_N$ which is 1-sided NTA with ADR boundary and satisfies the qualitative exterior corkscrew condition. Our main assumption is then
\begin{equation}\label{main-RT-approx}
\sup_{\eps>0}\int_{\pom}\left|\int_{\{y\in \pom:|x-y|>\eps\}}\frac{x-y}{|x-y|^{n+1}} \,f(y)\, dH^n(y)\right|^2 dH^n(x) \,\leq \,C_{\pom}
\int_{\pom}|f|^2 dH^n\,.
\end{equation}
We warn the reader that in this step the dyadic grid, sawtooth regions, Carleson boxes, etc., are with respect to $\Omega$ (which now equals $\Omega_N$) since this is our ambient space.

We proceed as in \cite{HM-I}, with a variation that incorporates harmonic measure
in certain places in lieu of surface measure.  We have already introduced some notation from \cite{HM-I}. However the reader may find convenient to have \cite[Sections 5, 6, 7, 8]{HM-I} handy for more details.

First, recalling the definition of the non-tangential maximal operator
\eqref{eqNTdef}, we note that the $L^2$ boundedness of the Riesz transform stated in \eqref{main-RT-approx} implies the $L^2$ boundedness of the operator  $f\mapsto N_*(\nabla\s f)$.
This estimate is a standard consequence of Cotlar's inequality for maximal singular integrals,
and we omit the proof.

Now fix a dyadic cube $Q_0$.
Define a local dyadic maximal function
$$M^{\rm dyadic}_{Q_0}f(x):= \sup_{Q:x\in Q\in\dd_{Q_0}}\fint_Q |f|\,d\sigma\,.$$

Recall that there exists a constant $K\ge 2$ depending on the ADR and 1-sided NTA constants of $\Omega$ only such that  for any $Q$ there is a ball  $B^*_{Q}$ of radius $K\ell(Q)$ such that $T_{Q}^{fat}\subset \frac14 B^*_{Q}$. We also denote by  $B^{**}_{Q_0}$ a fattened ball of radius $K'\ell(Q_0)$, with the value of $K'\geq K$, depending on the ADR and 1-sided NTA constants of $\Omega$ only, to be specified below.
Write $\Delta_0^*=B^{**}_{Q_0}\cap\pom$.
For $M$ a large constant to be chosen, set 
\begin{equation}\label{oodef}
\oo_0:=\{x\in Q_0: M^{\rm dyadic}_{Q_0}\left(N_*\left(\nabla \s 1_{2\Delta_0^*}\right)\right)(x)>M\}\,.
\end{equation}
Let us denote  by
$\F_0:=\{Q_j^0\}\subset \dd_{Q_0}$ the maximal family of disjoint cubes such that $\oo_0 =\cup_{\F_0}Q_j^0$. Note that  by the $L^2$ boundedness
of the operator $f\to N_*(\nabla\s f)$,
for $M$ chosen large enough we have
\begin{equation}\label{ample-sawtooth}
\sigma(Q_0\setminus \oo_0)=\sigma\big(Q_0\setminus \cup_{\F_0} Q_j^0\big)\ge
\big(1-C\,M^{-2}\big)\,\sigma(Q_0)=:
(1-\alpha)\,\sigma(Q_0)
\end{equation}
with $0<\alpha\ll 1$. We also note that by the maximality of the dyadic cubes in $\F_0$ and the Lebesgue differentiation theorem one can easily show that for every $Q\in\dd_{\F_0,Q_0}$
there exists $F_Q\subset Q$ with $\sigma(F_Q)>0$, such that
\begin{equation}\label{FQ}
N_*\left(\nabla \s 1_{2\Delta_0^*}\right)(y)\le M, \quad \forall\,y\in F_Q.
\end{equation}
This estimate implies that $|\nabla \s 1_{2\Delta_0^*}(X)|\leq M$ for all $X$ in the union of the ``cones'' $\Gamma(y)$ with $y\in F_Q$ and $Q\in\dd_{\F_0, Q_0}$. This covers a big portion of $X\in\Omega^{fat*}_{\F_0} \cap B^{**}_{Q_0}$ but it is not all of it (take for instance $X$ with $\delta(X)\ll \ell(Q_0)$ and $\dist(X,Q_0)\approx \ell(Q_0)$). This is a problem when obtaining the key Carleson estimate in Lemma \ref{lemma:Carleson}.
In order to solve this issue we are going to augment the family $\F_0$ by adding some neighbors of $Q_0$. Namely,  fix $\NN, \tau\gg 1$, to be chosen, and write
$$
\mathfrak{C}_{Q_0}
=
\big\{
Q\in\dd: \ell(Q)\le 2^{\NN}\ell(Q_0),\ Q\cap Q_0=\emptyset,\ Q^\NN\cap \tau B_{Q_0}^{**}\neq\emptyset
\big\}
$$
where $Q^\NN$ is the unique dyadic cube containing $Q$ with $\ell(Q)=2^\NN\ell(Q_0)$. We then let $\F_{\mathfrak{C}_{Q_0}}$ denote the collection of maximal cubes of $\mathfrak{C}_{Q_0}$. These cubes do not meet $Q_0$ by definition and therefore $\F_1:=\F_0\cup \F_{\mathfrak{C}_{Q_0}}$ is a family of pairwise disjoint cubes. Consider $\Omega^{fat}_{\F_1}$ and $\Omega^{fat*}_{\F_1}$ and note that
$$
\sigma(\partial \Omega^{fat}_{\F_1}\cap Q_0)=\sigma(\partial \Omega^{fat*}_{\F_1}\cap Q_0)
=
\sigma\big(Q_0\setminus \cup_{\F_0} Q_j^0\big)\ge (1-\alpha)\,\sigma(Q_0),
$$
that is, the sawtooth regions $\Omega^{fat}_{\F_1}$, $\Omega^{fat*}_{\F_1}$ have  ``ample" contact with $Q_0$.

We claim that
\begin{equation}\label{local-term-small}
|\nabla \s 1_{2\Delta_0^*}(X)|\leq M\,,\qquad X\in
\Omega^{fat*}_{\F_1} \cap B^{**}_{Q_0}.
\end{equation}
To prove that, we fix $X\in \Omega^{fat*}_{\F_1} \cap B^{**}_{Q_0}$, so that, in particular, $X\in U_Q^{fat*}$, $Q\in\dd_{\F_1}$.

\noindent{\bf Case 1}: $Q\in\dd_{Q_0}$. In this case we necessarily have $Q\in\dd_{\F_0}$ and, by definition, $U_Q^{fat*}\subset \Gamma(y)$ for any $y\in F_Q$.
Therefore \eqref{FQ} implies that $|\nabla \s 1_{2\Delta_0^*}(X)|\le N_*\left(\nabla \s 1_{2\Delta_0^*}\right)(y)\le M$.

\noindent{\bf Case 2}: $Q\notin\dd_{Q_0}$, $Q_0\subsetneq Q$. We first observe that $Q_0\in\dd_{\F_0,Q_0}$ and therefore \eqref{FQ} gives $N_*\left(\nabla \s 1_{2\Delta_0^*}\right)(y)\le M$ for every $y\in F_{Q_0}\subset Q_0$ and $\sigma(F_{Q_0})>0$. Then note that $U_Q^{fat*}\subset \Gamma(y)$ for every $y\in Q_0$, $Q\supset Q_0$, and consequently $|\nabla \s 1_{2\Delta_0^*}(X)|\le M$.

\noindent{\bf Case 3}: $Q\notin\dd_{Q_0}$, $Q_0\not\subsetneq Q$. In this case we have $Q\cap Q_0=\emptyset$. Note that since $X\in B^{**}_{Q_0}$ we have
$$
\ell(Q)\approx\delta(X)\le|X-x_{Q_0}|\lesssim \ell(Q_0)
$$
and taking $\NN$ large enough $\ell(Q)\le 2^{\NN}\ell(Q_0)$. Also,
by definition of $U_Q^{fat*}$ (cf. \eqref{UQ-fat*}), we have that $X\in 5I$
for some $I\in\,\mathcal{W}^*_Q$, so that
$$
|x_Q-x_{Q_0}|
\lesssim
\ell(Q)+\dist(I,Q)+\ell(I)+|X-x_{Q_0}|
\lesssim
\ell(Q_0).
$$
Thus $x_Q\in \tau B^{**}_{Q_0}$ if we take $\tau$ large enough, and, with the notation above,
$Q^\NN\cap \tau B^{**}_{Q_0}\neq\emptyset$. Thus $Q\in \mathfrak{C}_{Q_0}$ and by maximality there exists $Q^{max}\in\F_{\mathfrak{C}_{Q_0}}$ with $Q\subset Q^{max}$.
This contradicts the fact that $Q\in \dd_{\F_1}$ and this case is vacuous.

We next pick $\vec{C}_{Q_0}=\nabla\s 1_{\pom\setminus2\Delta_0^*}(x_{Q_0})$ and obtain by \eqref{local-term-small}  and standard Calde\-r\'{o}n-Zygmund estimates
that
\begin{equation}\label{linftybound}|\nabla \s 1(X) -\vec{C}_{Q_0}|\lesssim M\,,\qquad X\in
\Omega^{fat*}_{\F_1} \cap B^{**}_{Q_0}.
\end{equation}

As mentioned before, in the current step $\Omega$ is an approximating domain. Thus $\pom$ is comprised of pieces of faces of Whitney cubes of size $\approx 2^{-N}$.
Therefore,  qualitatively, $\omega\ll\sigma:=H^n\big|_{\pom}$.
On the other hand, as noted above, the fact that $\Omega$ is actually an approximating domain,
and therefore satisfies the qualitative exterior Corkscrew condition, implies that
$\hm$ is doubling and that the
Caffarelli-Fabes-Mortola-Salsa estimates hold,
with bounds that are independent of $N$ (cf.  \eqref{CFMS}-\eqref{doubling}).
We next deduce the following Carleson measure estimate:

\begin{lemma}\label{lemma:Carleson}
Under the previous assumptions
 \begin{equation}\label{carl}
\sup_{Q\in\dd_{Q_0}} \frac1{\hm^{X_0}(Q)}\iint_{T^{fat}_Q}|\nabla^2\s 1(X)|^2G(X,X_0)\,
1_{\Omega^{fat}_{\F_1}}(X)\,dX \leq \M_0\,,
\end{equation}
where the Carleson norm $\M_0$ depends on  $M$, and on dimension and the ADR and 1-sided NTA constants for $\Omega$. The point $X_0$ in \eqref{carl} is a corkscrew point for the surface ball $CB^{**}_{Q_0}\cap\pom$, where $C$ is chosen sufficiently large (depending on the ADR and 1-sided NTA constants of $\Omega$ only) to guarantee that $X_0\notin B^{**}_{Q_0}$.
\end{lemma}

\begin{proof}  We shall prove the lemma by a standard integration by parts argument, adapted to our setting, using
 \eqref{linftybound}, and the doubling property of harmonic measure.
Let us define a smooth cut-off function associated to $T^{fat}_Q$. Recall that $B^*_Q$ has radius $K l(Q)$ and we have  $T^{fat}_Q\subset \frac 14 B^*_Q$. We then take
$\Phi_{T^{fat}_Q}\in C_0^\infty(\RR^{n+1})$ with $\Phi_{T^{fat}_Q}=1$ on $ \frac 12 B^*_Q$ and $\Phi_{T^{fat}_Q}=0$ in $\Omega\setminus B^*_Q$ with $|\nabla^k \Phi_{T^{fat}_Q}|\lesssim l(Q)^{-k}$.
Taking $K'$ in the definition of $B^{**}_{Q_0}$ sufficiently large, one can ensure that for any $Q\subset Q_0$ we have $B^*_Q\subset \frac 12 B^{**}_{Q_0}$ and moreover, all regions $U_{Q'}^{fat*}$ that meet $B^*_Q$ are strictly inside $B^{**}_{Q_0}$. Now fix some $Q\subset Q_0$. Then
\begin{multline}\label{eqIBP1}
\iint_{T^{fat}_Q}|\nabla^2\s 1(X)|^2G(X,X_0)\,
1_{\Omega^{fat}_{\F_1}}(X)\,dX\\[4pt]\leq \iint_{\Omega^{fat}_{\F_1}} |\nabla^2\s 1(X)|^2G(X,X_0)\,
\Phi_{{T^{fat}_Q}}(X)\,dX\\[4pt]\approx \iint_{\Omega^{fat}_{\F_1}} \mathcal{L}\left( |\nabla \s 1(X)-\vec{C}_{Q_0}|^2\right)\,G(X,X_0)\,
\Phi_{{T^{fat}_Q}}(X)\,dX
\end{multline}
where $\mathcal{L}:=\nabla\cdot\nabla$ denotes the usual Laplacian in $\ree$. Recall that $\Omega$ is one of the approximating domains $\Omega_N$. Its boundary consists of portions of faces of fattened Whitney cubes of side length roughly $2^{-N}$. Thus, the outward unit normal $\nu$ is well-defined a.e. on $\pom$, as well as on $\partial{\Omega^{fat}_{\F_1}}$,  and we can apply the divergence theorem to integrate by parts.
Let us write $\sigma_\star=H^n\big|_{\partial\Omega^{fat}_{\F_1}}$.
Since $X_0\notin \overline{B^*_Q\cap \Omega^{fat}_{\F_1}}$, we have that $G(X,X_0)$ is harmonic in $\overline{B^*_Q\cap \Omega^{fat}_{\F_1}}\cap\Omega$, and of course vanishes on $\pom$.
Therefore, by Green's formula,
\begin{equation} \label{main-Carl}
\iint_{T^{fat}_Q}|\nabla^2\s 1(X)|^2G(X,X_0)\,
1_{\Omega^{fat}_{\F_1}}(X)\,dX
\lesssim
I+II+III
\end{equation}
where
$$
I=\int_{\pom\cap\partial \Omega^{fat}_{\F_1}} |\nabla \s 1(x)-\vec{C}_{Q_0}|^2 \,\Phi_{{T^{fat}_Q}}(x) \,d\omega^{X_0}(x),
$$
\begin{multline*}
II=\iint_{\Omega^{fat}_{\F_1}} |\nabla \s 1(X)-\vec{C}_{Q_0}|^2\Big(|\nabla G(X,X_0)|\,
|\nabla \Phi_{{T^{fat}_Q}}(X)|
\\
+|G(X,X_0)|\,
|\nabla^2 \Phi_{{T^{fat}_Q}}(X)|\Big)dX,
\end{multline*}
and
\begin{multline*}
III=
\int_{\partial \Omega^{fat}_{\F_1}\setminus \pom} \Big(
\big|\partial_\nu\big(|\nabla \s 1(x)-\vec{C}_{Q_0}|^2\big)\big|\,|G(x,X_0)|\,
| \Phi_{{T^{fat}_Q}}(x)|
\\
+|\nabla \s 1(x)-\vec{C}_{Q_0}|^2\,|G(x,X_0)|\,|\partial_\nu  \Phi_{{T^{fat}_Q}}(x)|
\\
+ |\nabla \s 1(x)-\vec{C}_{Q_0}|^2\,|\partial_\nu G(x,X_0)|\,
| \Phi_{{T^{fat}_Q}}(x)|\Big)\,d\sigma_\star(x)\,.
\end{multline*}

We start with $I$.  Since we are working in an approximating domain,
the estimate \eqref{linftybound}
carries all the way to $\pom\cap
\partial \Omega^{fat}_{\F_1}$ (a.e.),
by standard trace theory for layer potentials.   Moreover $\hm\ll\sigma$ (qualitatively, in the approximating domains), as noted above.
Hence,
\begin{equation}\label{I}
I \lesssim  M^2 \int_{\Delta^*_Q} d\omega^{X_0}(X) \lesssim M^2 \omega^{X_0}(Q),
\end{equation}
where as usual, $\Delta^*_Q:=B^*_Q\cap\pom$, and we have used the doubling property \eqref{doubling} to obtain the last inequality.


For $II$, we use \eqref{linftybound} and interior estimates for harmonic functions to write
\begin{align*}
II
&
\lesssim
M^2\iint_{B_Q^*} \left(\frac{G(X,X_0)}{\delta(X)\ell(Q)}+\frac{G(X,X_0)}{\ell(Q)^2}\right)\,dX
\lesssim
M^2\iint_{B_Q^*} \frac{G(X,X_0)}{\delta(X)\ell(Q)}\,dX.
\end{align*}
Define $\mathcal{I}^*_Q=\{I\in\mathcal{W}: I\cap B_Q^*\neq\emptyset\}$. For each $I\in \mathcal{I}^*_Q$ we let $Q_I\in\dd$ be such that $\ell(Q_I)=\ell(I)$ and $\dist(I,Q_I)\approx \ell(I)$ (take for instance $Q_I$ that contains $x_I\in\pom$ where $\delta(X(I))=|X(I)-x_I|$ with $X(I)$ being the center of $I$). Note that $Q_I\subset C\,\Delta_Q^*$ and also that if $Q_I=Q_{I'}$ then $\dist(I,I')\lesssim\ell(I)=\ell(I')$. This implies that the family $\{Q_I:\ell(I)=2^{-k}\}$ has bounded overlap. Then by the Harnack chain condition, \eqref{CFMS} and \eqref{doubling} we obtain
\begin{multline}\label{II}
II
\lesssim
M^2\sum_{I\in \mathcal{I}^*_Q}
\frac{G(X_I,X_0)\,\ell(I)^n}{\ell(Q)}
\lesssim
M^2\sum_{I\in \mathcal{I}^*_Q}
\frac{\ell(I)}{\ell(Q)} \hm^{X_0}(Q_I)
\\
=M^2
\sum_{k:2^{-k}\lesssim \ell(Q)}\frac{2^{-k}}{\ell(Q)}\sum_{I\in \mathcal{I}^*_Q:\ell(I)=2^{-k}} \hm^{X_0}(Q_I)
\lesssim
M^2
\hm^{X_0}(C\,\Delta_Q^*)
\lesssim
M^2\hm^{X_0}(Q).
\end{multline}

Finally, we estimate $III$.
Note that by construction,
\eqref{linftybound} gives $|\nabla \s 1(x) -\vec{C}_{Q_0}|\lesssim M$ for every $x\in  \big(\partial\Omega^{fat}_{\F_1}\setminus\pom\big)\cap B^*_Q$.  Moreover, by \eqref{linftybound} and interior estimates for harmonic functions,
for $x\in\big(\partial\Omega^{fat}_{\F_1}\setminus\pom\big)\cap B^*_Q$ we also have
$$
\left|\partial_\nu\left(|\nabla \s 1(X)-\vec{C}_{Q_0}|^2\right)\right|\lesssim |\nabla(\nabla \s 1(X)-\vec{C}_{Q_0})|\, |\nabla \s 1(X)-\vec{C}_{Q_0}|
\lesssim \frac{M^2}{\delta(x)}\,.
$$
Then proceeding as before
\begin{multline*}
III
\lesssim
M^2\,\int_{\big(\partial\Omega^{fat}_{\F_1}\setminus\pom\big)\cap B^*_Q} \left(\frac{G(x,X_0)}{\delta(x)}+ \frac{G(x,X_0)}{\ell(Q)}\right)\,d\sigma_\star(x)
\\
\lesssim
M^2\,\int_{\big(\partial\Omega^{fat}_{\F_1}\setminus\pom\big)\cap B^*_Q} \frac{G(x,X_0)}{\delta(x)}\,d\sigma_\star(x).
\end{multline*}

Let us briefly state the strategy to estimate the last term. Notice that $\Sigma=\Omega^{fat}_{\F_1}\setminus\pom$ consists of portions of faces of some fattened Whitney cubes. Take one of these Whitney cubes, say $J$, and choose some  dyadic cube $Q_J$ such that $\ell(Q_J) \approx\ell(J)$ and $\dist(J,Q_J)\approx \ell(J)$. Then \eqref{CFMS} imply
$$
\int_{\Sigma\cap J} \frac{G(x,X_0)}{\delta(x)}\,d\sigma_\star(x)
\lesssim
\frac{\hm^{X_0}(Q_J)}{\ell(Q_J)^n} \,H^n(\Sigma\cap J)
\lesssim
\hm^{X_0}(Q_J).
$$
From here it remains to sum in $J$. For that, we will need to have some bounded overlap property, which we shall obtain by choosing a particular collection of $Q_J$'s. Let us make this precise by using some ideas from \cite[Appendix A.3]{HM-I}.  First, we observe that $\Sigma=\partial\Omega^{fat}_{\F_1}\setminus\pom$ consists of (portions of) faces of certain fattened Whitney cubes
$4J$, with $\interior(4J)\subset\Omega^{fat}_{\F_1}$,
which meet some $I\in\W$ such that $I\notin \W_Q^*$ for any $Q \in \dd_{\F_1}$.
If for each such $I$, we let $Q^*_I$
denote the nearest dyadic cube to $I$ with
$\ell(I) =\ell(Q^*_I)$,
then $I\in \W_{Q^*_I}^*$, 
whence it follows that $Q^*_I\subset Q_j$, for
some  $Q_j\in\F_1$.
On the other hand, since $\interior(4J)\subset\Omega^{fat}_{\F_1}$ we have that $J\in \W_{Q_J}^*$ for some $Q_J\in \dd_{\F_1}$. In, particular $Q_J$ is not contained in $Q_j$ and therefore, upon a moment's reflection, one may readily see that $\dist(Q^*_I,\pom\setminus Q_j)\lesssim\ell(Q^*_I)$.
Thus, as an easy consequence of the properties of the dyadic cubes,
we may select a descendant of $Q_I^*$, call it $Q_I$, of comparable size, in such a way that
\begin{equation}\label{eq2.39}
\dist(Q_I,\pom\setminus Q_j)\approx \ell(I)\approx\ell(Q_I)\,,
\end{equation}
while of course retaining the property that $\dist(Q_I,I)\approx \ell(I)$.

For each $Q_j\in \F_1$, we set
$$
\R_{Q_j}:=\cup_{Q'\in \dd_{Q_j}}\W_{Q'}^*,
$$
and denote by $\F_1^*$
the sub-collection of those $Q_j\in \F_1$ such that
there is an $I\in \R_{Q_j}$ for which $B_Q^*\cap\Sigma$ meets
$I $. For a given $I\in \R_{Q_j}$, the relevant $Q'\in \dd_{Q_j}$, such that
$I\in \W_{Q'}^*$,
is $Q'=Q^*_I$.  Set $\Sigma_j:=\Sigma\cap(\cup_{I\in\R_{Q_j}} I)$ and note that $\Sigma\subset \cup_j \Sigma_j$.
We observe that by \eqref{eq2.39}, for $j$ fixed, the cubes $Q_I$ have bounded overlaps
as the index runs over those $I\in \R_{Q_j}$ which meet $\Sigma$.
Indeed, \eqref{eq2.39} says that two cubes $Q_I$ and $Q_{I'}$ cannot meet
unless $\ell(I)\approx\ell(I')$, and by a simple geometric packing argument, the cubes $Q_I$, corresponding to a collection $\{I\}$ of
Whitney boxes  having comparable size, clearly have bounded overlaps.   We note that if
$x\in I\cap\Sigma_j\cap B_Q^*$, then by the Harnack chain condition and \eqref{CFMS} we have
$$
\frac{G(x,X_0)}{\delta(x)}
\lesssim
\frac{\hm^{X_0}(Q_I)}{\ell(I)^n}\,.
$$
Moreover, $\sigma_\star (I\cap \Sigma)=H^n(I\cap \Sigma)\lesssim \ell(I)^n$, since
$\Sigma\cap I$ consists of a bounded number  of (portions of) faces of certain fattened Whitney cubes with side length of the order of $\ell(I)$.
Thus,  for every $j$ such that $Q_j\in \F_1^*$
\begin{multline}\label{claim:bd}
\int_{\Sigma_j\cap B^*_Q} \frac{G(x,X_0)}{\delta(x)}\,d\sigma_\star(x)
\leq \sum_{I\in\R_{Q_j}} \int_{\Sigma\cap I\cap B^*_Q} \frac{G(x,X_0)}{\delta(x)}\,d\sigma_\star(x)\\[4pt]
\lesssim  \sum_{I\in\R_{Q_j}} \hm^{X_0}(Q_I) \lesssim
\hm^{X_0}(Q_j\cap C\Delta_Q^*)\,,
\end{multline}
by the bounded overlap property of the $Q_I$'s.
It follows that
\begin{multline}\label{III}
III
\lesssim M^2
\sum_{j:\,Q_j\in\F^*_1}\int_{\Sigma_j\cap B^*_Q} \frac{G(x,X_0)}{\delta(x)}\,d\sigma_\star(x)
\\
\lesssim M^2
\sum_{j:\,Q_j\in\F^*_1}\hm^{X_0}(Q_j\cap C\Delta_Q^*)
\lesssim
M^2 \hm^{X_0}(Q),
\end{multline}
where in the last inequality we have used 
the pairwise disjointness of the $Q_j$'s, 
and also the doubling property \eqref{doubling}.
We may then plug \eqref{I}, \eqref{II} and \eqref{III} into \eqref{main-Carl}, to conclude as desired that \eqref{carl} holds, with $\M_0\approx M^2$.  \end{proof}

We now define
\begin{equation}\label{eq7.2a}
\alpha_Q := \dint_{U_Q^{fat}}|\nabla^2 \mathcal{S}1(X)|^2 \,G(X,X_0) \,dX \,,\qquad
Q\in\dd_{\F_0,Q_0}\,,\end{equation}
and set $\alpha_Q\equiv 0$, for all $Q\in \dd_{Q_j^0}$, and for every $Q_j^0\in\F_0.$
Given any sub-collection $\dd'\subset\dd_{Q_0}$, we define
\begin{equation}\label{eq7.3a}\mut(\dd'):= \sum_{Q\in\dd'}\alpha_Q\,.
\end{equation}
By \eqref{carl}, $\mut$ satisfies the discrete Carleson measure condition
\begin{equation*}
\mut(\dd_Q)\lesssim\M_0 \,\hm^{X_0}(Q)\,, \qquad Q\in\dd_{\F_0,Q_0}\,,
\end{equation*}
(cf. \cite[Section 8.1]{HM-I})
and $\mut(\dd_Q) =0$, for all $Q$ contained in any $Q_j^0\in\F_0$.
Since $\mut$ is non-zero only in $\dd_{\F_0,Q_0}$, we have that
\begin{equation}\label{carldiscrete}
\mut(\dd_Q)\lesssim\M_0 \, \omega_0(Q)\,,\qquad Q\in\dd_{Q_0}\,,
\end{equation}
where
\begin{equation}\label{eq2.hm0def}
\hm_0:= \Pfo^{\,\sigma}\,\hm^{X_0}\,,
\end{equation}
and in general, the projection of a Borel measure $\mu$
with respect to a pairwise disjoint family $\F:=\{Q_j\}\subset\dd$, in terms of another Borel measure
$\nu$, is defined by
\begin{equation}\label{projdef}
\Pf^{\,\nu}\mu(A):= \mu (A\setminus \cup_\F Q_j)+\sum_\F \frac{\nu(A\cap Q_j)}{\nu (Q_j)}\,\mu(Q_j).
\end{equation}
In particular, we have that $\P_\F^{\,\nu}\mu(Q)=\mu(Q)$, for every $Q\in\dd_\F$ (i.e.,
for $Q$ not contained in any $Q_j\in\F$), and also that
$\P_\F^{\,\nu} \mu(Q_j)=\mu(Q_j)$ for every $Q_j\in\F$.

We now claim that $\hm_0\in A^{\rm dyadic}_\infty(Q_0)$ with respect to $\sigma$, that is, for every $Q\in\dd_{Q_0}$ and
$F\subset Q$, we have
\begin{equation}\label{eq1.ainftydyadic}
\frac{\hm_0 (F)}{\hm_0 (Q)}\leq C \left(\frac{\sigma(F)}{\sigma(Q)}\right)^\theta.
\end{equation}

Let us momentarily accept this claim, and show that it implies that $\pom$ is UR.  To this end, since $Q_0\in\dd$ was arbitrary, 
it is enough to show that $\hm^{X_0} \in A^{\rm dyadic}_\infty (Q_0)$
uniformly in $Q_0$:
indeed the latter property plus \eqref{doubling} imply that $\hm^{X_\Delta}\in A_\infty(\Delta)$ for every surface ball $\Delta$, whence it follows immediately by \cite[Theorem 1.22]{HMU}
that $\pom$ is
UR.  In turn, again since $Q_0\in\dd$ is arbitrary,
 by the arguments of Bennewitz and Lewis \cite{BL} (for which, one may also consult
\cite[Section 8]{HM-I}),  to establish the dyadic $A_\infty$ property,
it suffices to show that there are constants $\eta_0,c_0\in (0,1)$  such that
for every Borel set $A\subset Q_0$, we have
\begin{equation}\label{*}
\sigma(A)>(1-\eta_0)\,\sigma(Q) \implies \hm^{X_0}(A)\geq c_0\,.
\end{equation}
To prove this, we observe that, given such an $A$, for uniform choices of
$\eta_0$ small enough, and $M$  (in \eqref{oodef}) large enough, we have that \eqref{ample-sawtooth} implies
$\sigma(A\setminus \oo_0) \geq c_1 \sigma(Q_0)$
for some uniform constant $c_1>0$. Therefore $\omega_0(A\setminus \oo_0)\gtrsim \omega_0(Q_0)$, since our claimed dyadic $A_\infty$ property
\eqref{eq1.ainftydyadic} implies immediately the
converse of itself.  The latter fact
is standard (see \cite[Lemma B.7]{HM-I} in the present context) since $\sigma$ and $\hm_0$ are dyadically doubling (for the latter we invoke \cite[Lemma B.1]{HM-I} and \eqref{doubling}). It therefore follows that
$$\hm^{X_0}(A) \geq \hm^{X_0}(A\setminus \oo_0)= \omega_0(A\setminus \oo_0)
\gtrsim \omega_0(Q_0) = \hm^{X_0}(Q_0) \approx 1\,,$$
where we have used (in the two equalities) the definition of the projection operator $\Pfo^{\,\sigma}$,
 and finally Bourgain's estimate (cf. \eqref{eq2.Bourgain2}).  Thus, \eqref{*} holds, and as explained above, this concludes the proof of the uniform rectifiability of $\pom$, modulo the claim \eqref{eq1.ainftydyadic}.

We now turn to the proof of \eqref{eq1.ainftydyadic}.  We first record a few preliminary observations.
Recall that we are working in an approximating domain, so that $\hm^{X_0}$ is doubling, with
a {\it uniform} constant that depends {\it only} upon dimension and the constants in the ADR and
1-sided NTA conditions.
By  \cite[Lemma B.1]{HM-I}, we then have that $\hm_0$ is dyadically doubling (again with the same dependence).
Our strategy is to use  \cite[Lemma 8.5]{HM-I}  but with the roles of
$\sigma$ and $\omega$ reversed, and in our case, with $\hm = \hm_0$. We note that \cite[Lemma 8.5]{HM-I}
is a purely real variable result, and the only requirement on the two measures involved is that they be
non-negative, dyadically doubling Borel measures (see \cite[Remark 8.9]{HM-I}).  This role reversal of $\sigma$ and $\hm_0$ is required by the fact that our discrete Carleson condition \eqref{carldiscrete} is expressed in terms
of $\hm_0$. To be precise we state that result in the formulation needed here:

\begin{lemma}[See {\cite[Lemma 8.5, Remark 8.9]{HM-I}}]\label{lemma:extrapol}
We fix $Q_0\in \dd$.
Let $\sigma$ and $\omega_0$ be a pair of
non-negative,
dyadically doubling Borel measures on $Q_0$,
and let $\mut$ be a discrete Carleson measure with respect to $\hm_0$
(cf. \eqref{carldiscrete}) with
$$
\sup_{Q'\in \dd_{Q_0}}\frac{\mut(\dd_{Q'})}{\hm_0(Q')} \leq \mathcal{M}_0.
$$
Suppose that there is a $\gamma>0$ such that for every $Q\in \dd_{Q_0}$ and every family of pairwise disjoint dyadic subcubes  $\F=\{Q_j\}\subset \dd_{Q}$
verifying
\begin{equation}\label{smallmut}
\sup_{Q'\in \dd_{\F,Q}}\frac{\mut(\dd_{Q'})}{\hm_0(Q')} \leq \gamma\,,
\end{equation}
we have that $\Pf^{\hm_0}\sigma$ satisfies the following property:
\begin{equation}\label{extrap:Ainfty:Pw}
\forall\,\varepsilon\in (0,1),\ \exists\, C_\varepsilon>1 \mbox{ such that }
\Big(
F\subset Q,\ \ \frac{\hm_0(F)}{\hm_0(Q)}\ge
\varepsilon\quad \Longrightarrow \quad
\frac{\P^{\hm_0}_\F \,\sigma(F)}{\P^{\hm_0}_\F\, \sigma(Q)}\ge \frac1{C_\varepsilon}\Big).
\end{equation}
Then, there exist $\eta_0\in(0,1)$  and $C_0<\infty$
such that, for every  $Q\in \dd_{Q_0}$,
\begin{equation}\label{extrap:Ainfty:w}
F\subset Q,\quad \frac{\hm_0(F)}{\hm_0(Q)}\ge
1-\eta_0\quad \Longrightarrow \quad \frac{\sigma(F)}{\sigma(Q)}\ge \frac1{C_0}.
\end{equation}
\end{lemma}

\begin{remark}
Let us point out that \eqref{extrap:Ainfty:w} says that $\sigma \in A^{\rm dyadic}_\infty(Q_0,\hm_0)$. This in turn is equivalent to $\hm_0 \in A^{\rm dyadic}_\infty(Q_0,\sigma)$ since $\sigma$ and $\hm_0$ are dyadically doubling, see \cite[Lemma B.7]{HM-I}.
\end{remark}

To show that \eqref{smallmut} implies \eqref{extrap:Ainfty:Pw} we record for the reader's convenience that
\begin{align}\label{projdef2}
\Pf^{\,\hm_0}\sigma(F):= \sigma (F\setminus \cup_\F Q_j)
+\sum_\F \frac{\hm_0(F\cap Q_j)}{\hm_0 (Q_j)}\,\sigma(Q_j)\,;
\\[4pt]
\label{projdef3}
\Pf^{\,\sigma}\hm_0(F):= \hm_0 (F\setminus \cup_\F Q_j)
+\sum_\F \frac{\sigma(F\cap Q_j)}{\sigma (Q_j)}\,\hm_0(Q_j)
\end{align}
(cf. \eqref{projdef}).
We next note that
\eqref{extrap:Ainfty:Pw} holds trivially, with $Q$ replaced by any $\widetilde{Q}$
contained in some $Q_j\in \F$.
Indeed, in that case we have by \eqref{projdef2}, that for $F\subset \widetilde{Q}\subset Q_j\in\F$,
$$\frac{\Pf^{\,\hm_0}\sigma(F)}{\Pf^{\,\hm_0}\sigma(\widetilde{Q})}=
\frac{\frac{\hm_0(F)}{\hm_0 (Q_j)}\,\sigma(Q_j)}{\frac{\hm_0(\widetilde{Q})}{\hm_0 (Q_j)}\,\sigma(Q_j)}
=\frac{\hm_0(F)}{\hm_0(\widetilde{Q})}\,.$$
Moreover,
\eqref{smallmut} holds also with $Q$ replaced by any $\widetilde{Q}\in \dd_{\F,Q}$.
Therefore, by the usual self-improvement properties for (dyadic) $A_\infty$ weights,
it is enough to show that, given \eqref{smallmut} with $\gamma$ small enough,
there are constants $\eta_1,c_1\in(0,1)$ such that
\begin{equation}\label{eq1.13+}
F\subset Q,\ \ \frac{\hm_0(F)}{\hm_0(Q)}\ge
1-\eta_1\quad \Longrightarrow \quad
\frac{\P^{\hm_0}_\F \,\sigma(F)}{\P^{\hm_0}_\F\, \sigma(Q)}\ge c_1\,,
\end{equation}
since the same will then hold also for every $\widetilde{Q}\in\dd_{\F,Q}$.
In turn, to verify the latter implication, our strategy will be to prove that,
given \eqref{smallmut} with $\gamma$ small enough, we have
\begin{equation}\label{eq1.14+}
\Pf^\sigma \hm_0 \in A_\infty^{\rm dyadic}(Q), \,{\rm \,with \,respect\, to\,} \sigma\,.
\end{equation}
We defer momentarily the proof of this fact, and we show that
\eqref{eq1.14+} implies \eqref{eq1.13+}.  There are two cases.
We now fix $\eta_1$ so that $1-3\eta_1> 3/4$,
and suppose that $\hm_0(F)\ge (1-\eta_1)\,\hm_0(Q)$.

\medskip
\noindent{\bf Case 1:}  $\hm_0(Q\setminus \cup_{\F}Q_j)\geq 2\eta_1 \hm_0(Q).$  In this case,
by definition \eqref{projdef3},
$$\Pf^{\,\sigma}\hm_0(F\setminus \cup_{\F}Q_j)
=\hm_0(F\setminus \cup_{\F}Q_j) \geq \eta_1\,\hm_0(Q)=\eta_1\,\Pf^{\,\sigma}\hm_0(Q)\,.$$
Consequently, by  \eqref{eq1.14+},
we have
$$\sigma(F\setminus\cup_{\F}Q_j) \geq C_{\eta_1} \sigma(Q).$$
We now obtain
\eqref{eq1.13+} from properties of the projection operator \eqref{projdef2}.

\medskip

\noindent{\bf Case 2:}  $\hm_0(Q\setminus \cup_{\F}Q_j)\leq 2\eta_1 \hm_0(Q).$  In this case,
we then have that
$$\sum_{\F}\hm_0(Q_j) \geq (1-2\eta_1)\,\hm_0(Q)\,.$$ Set $\F':= \big\{Q_j\in\F: \hm_0(F\cap Q_j)\geq \frac13
\hm_0(Q_j)\big\}$.
It follows that
\begin{equation}\label{eq1.16+}
\sum_{\F'}\hm_0(Q_j) \geq \frac13 \hm_0(Q)\,,\end{equation}
for if not, we would have
\begin{multline*}(1-\eta_1)\hm_0(Q)\leq\hm_0(F) = \sum_{\F'} \hm_0(F\cap Q_j) + \sum_{\F\setminus\F'} \hm_0(F\cap Q_j)
+\hm_0(F\setminus
\cup_{\F} Q_j)\\[4pt]
\leq\, \frac23\hm_0(Q) + 2\eta_1\hm_0(Q)\,,
\end{multline*}
which contradicts that
we have chosen $1-3\eta_1>3/4$.  Moreover, by definition of the projections,
we have that for each $Q_j\in\F$, $\Pf^{\,\sigma}\hm_0(Q_j) = \hm_0(Q_j)$,
so by \eqref{eq1.14+}
and \eqref{eq1.16+}, we have
$$\sum_{\F'}\sigma(Q_j)\geq c\, \sigma(Q)=c\,\Pf^{\hm_0}\sigma(Q)\,.$$
Thus,
$$\P^{\hm_0}_\F \,\sigma(F) \geq \sum_{\F'} \frac{\hm_0(F\cap Q_j)}{\hm_0 (Q_j)}\,\sigma(Q_j)
\geq \frac13 \sum_{\F'}\sigma(Q_j)\,\gtrsim \,\Pf^{\hm_0}\sigma(Q)\,,$$
as desired.

It therefore remains only to verify \eqref{eq1.14+}, assuming that \eqref{smallmut} holds for
$\gamma$ sufficiently small.  Let $\F_2$ denote the maximal elements of $\F \cup \F_0$.
By definition of $\hm_0$, $\P_\F^\sigma \hm_0$ has a constant density with respect to $\sigma$, on any cube $Q$ that is contained in any $Q_j\in\F_2$; thus, in proving \eqref{eq1.14+}, we may henceforth assume that $Q\in\dd_{\F_2,Q_0}$.
As in \cite[Section 8.1]{HM-I}, it follows from  \eqref{smallmut}
that
\begin{equation}\label{carlsmall}
\sup_{Q'\in\dd_Q} \frac1{\hm^{X_0}(Q')}\iint_{\Omega^{fat}_{\F_2,Q'}}|\nabla^2\s 1(X)|^2G(X,X_0)\,
dX \leq C\,\gamma\,.
\end{equation}
We proceed as in \cite[Proof of Lemma 5.10]{HM-I} and use the notation there.
Let us fix $Q'\in \dd_{\F_2,Q}$, and
let $\Phi\in\C_0^\infty(\ree)$ be a smooth cut-off adapted to $Q'$.  Let $r\approx \ell(Q')$, and
let $\mathcal{L}:=\nabla\cdot\nabla$, as before, denote the Laplacian in $\ree$.  By \cite[Lemma 3.55]{HM-I} there exists an $(n+1)$-dimensional ball $B'_{Q'}$ of radius comparable to $r$  such that $B'_{Q'}\subset B_{Q'}$, $B'_{Q'}\cap\Omega\subset T_{Q'}$, and $\Omega_{\F_2,Q}\cap B'_{Q'}=\Omega_{\F_2,Q'}\cap B'_{Q'}$. We then have that
by the ADR property,
\begin{multline}\label{eq4.main} \hm^{X_0}(Q') \approx \frac{\hm^{X_0}(Q')}{\sigma(Q')} r^n
\approx \frac{\hm^{X_0}(Q')}{\sigma(Q')}
\int_{\partial\Omega} \Phi\, d\sigma
=\frac{\hm^{X_0}(Q')}{\sigma(Q')}\,\,\langle -\mathcal{L}\,\mathcal{S}1,\Phi\rangle\\[4pt]
=\frac{\hm^{X_0}(Q')}{\sigma(Q')} \iint_{\ree}\Big(\nabla \mathcal{S}1(X)-
\nabla \mathcal{S}1_{(2\Delta^*_{Q'})^c}(x_{Q'})- \vec{\alpha}\Big)\cdot \nabla\Phi(X)\,dX\\[4pt]
\lesssim \,\frac{\hm^{X_0}(Q')}{\sigma(Q')}\, \frac1r\, \iint_{B'_{Q'}}\big|\nabla \mathcal{S}1(X)-
\nabla \mathcal{S}1_{(2\Delta^*_{Q'})^c}(x_{Q'})- \vec{\alpha}\big|\,dX\\[4pt]
\,=\,\frac{\hm^{X_0}(Q')}{\sigma(Q')}\,\frac1r\left(\iint_{\Omega\cap B'_{Q'}}
\,\,+\,\,\iint_{\Omega_{ext}\cap B'_{Q'}}\right)\\[4pt]
=\,\frac{\hm^{X_0}(Q')}{\sigma(Q')}\,\frac1r\left(\iint_{\Omega_{\F_2,Q'}\cap B'_{Q'}}\,\,+
\,\,\iint_{\big(\Omega\setminus \Omega_{\F_2,Q}\big)\cap B'_{Q'}}
\,\,+\,\,\iint_{\Omega_{ext}\cap B'_{Q'}}\right)\\[4pt]
=:\, \frac{\hm^{X_0}(Q')}{\sigma(Q')}\,\frac1r\Big(I+II+III\Big)\,,
\end{multline}
where $\vec{\alpha}$ is a constant vector at out disposal and  $\Omega_{ext}:= \ree\setminus \overline{\Omega}.$

We deal with term $I$ first.  Let $\eps >0$ be a small number to be determined and observe that
$$I=\iint_{\Omega_{\F_2(\eps r),Q'}\cap B'_{Q'}}\,\,+
\,\,\iint_{\big(\Omega_{\F_2,Q}\setminus \Omega_{\F_2(\eps r),Q'}\big)\cap B'_{Q'}}=:I' + I''\,,$$
where
in $\Omega_{\F_2(\eps r),{Q'}}$ we have that $\delta(X)\gtrsim \eps r$, and where
$\Omega_{\F_2,Q}\setminus \Omega_{\F_2(\eps r),Q'}$ is a thin ``collar region" of thickness
$\lesssim \eps r$ (see \cite[Section 4]{HM-I}).  As in  \cite[Proof of Lemma 5.10]{HM-I} we next pick
$$
\vec{\alpha}:= \frac1{|\Omega_{\F_2(\eps r),Q'}|}
\iint_{\Omega_{\F_2(\eps r),Q'}}\Big(\nabla \mathcal{S}1(X)-
\nabla \mathcal{S}1_{(2\Delta_{Q'}^*)^c}(x_{Q'})\Big)\,dX$$ and it is shown in \cite[Proof of Lemma 5.10]{HM-I} that $|\vec{\alpha}|\le C$.

Term $I''$ is estimated as in \cite[Proof of Lemma 5.10]{HM-I} by means of an $L^q$ bound for $\nabla \s 1_{\kappa Q'}$:
$$\frac{\hm^{X_0}(Q')}{\sigma(Q')}\,\frac1r\,I'' \lesssim \eps^\beta \,\hm^{X_0}(Q')\,,$$
for some $\beta>0$,
which may be hidden by choice of $\eps$ small enough.  For such an $\eps$ now fixed,
we now turn to term $I'$, bearing in mind that in the domain of integration in this term,
we have that $r \approx \delta(X)$, with implicit constants depending on $\eps$.  We use the Poincar\'{e} inequality proved in \cite[Lemma 4.8]{HM-I}:
$$
\iint_{\Omega_{\F_2(\epsilon r),Q'}}\big|\nabla \mathcal{S}1(X)-
\nabla \mathcal{S}1_{(2\Delta^*_{Q'})^c}(x_{Q'})- \vec{\alpha}\big|^2dX \leq C_{\epsilon}\,
r^2\iint_{\Omega^{fat}_{\F_2(\epsilon r),Q'}}|\nabla^2\s 1(X)|^2dX.
$$
By this estimate, Cauchy-Schwarz, \eqref{CFMS},  Harnack, and \eqref{carlsmall},
we therefore obtain that
\begin{multline*}
\null\!\!\!\!\!\frac{\hm^{X_0}(Q')}{\sigma(Q')}\,\frac1r\, I'
\leq\,C_{\eps} \left(\hm^{X_0}(Q')\right)^{1/2}
\left(\frac{G(X_{Q'},X_0)}{r}\iint_{\Omega^{fat}_{\F_2(\eps r),Q'}}|\nabla^2\s 1(X)|^2\,\delta(X)\,dX\right)^{1/2}
\\[4pt] \leq\,C_{\eps}\left(\hm^{X_0}(Q')\right)^{1/2}
\left(\iint_{\Omega^{fat}_{\F_2(\eps r),Q'}}|\nabla^2\s 1|^2\,G(X,X_0)\,dX\right)^{1/2}
\leq C_{\eps}\,\gamma^{1/2}\, \hm^{X_0}(Q')\,,
\end{multline*}
which also may be hidden, by choice of $\gamma$ small enough, depending of course on $\eps$.
Thus, after hiding term I, and canceling $\hm^{X_0}(Q')$ in \eqref{eq4.main},
we have that $r^{n+1} \lesssim II+III$. Moreover, since $\nabla \s 1_{\kappa Q'}$ enjoys an  $L^q$ estimate, as in \cite[Proof of Lemma 5.10]{HM-I}, we have that
$$II + III \lesssim |\left(\Omega_{\F_2,Q}\right)_{ext}\cap B'_{Q'}|^{1/q'}r^{(n+1)/q}\,+
\,|\left(\Omega_{\F_2,Q}\right)_{ext}\cap B'_{Q'}|\,,$$
and therefore
$$r^{n+1} \lesssim |\left(\Omega_{\F_2,Q}\right)_{ext}\cap B'_{Q'}|\,.$$

Thus, by \cite[Lemma 5.7, proof of Lemma 5.10]{HM-I},  $\Omega_{\F_2,Q}$ satisfies a two-sided Corkscrew condition, at all scales $\lesssim \ell(Q)$.
Consequently, by \cite{DJe},  we have that $\tom^{X_Q}$,
the harmonic measure for $\Omega_{\F_2,Q}$, with pole at $X_Q$,
belongs to $A_\infty(\pom_{\F_2,Q})$, with respect to surface measure $\ts:= H^n|_{\Omega_{\F_2,Q}}$.
Here, $X_Q$ is a common Corkscrew point for $\Omega$, with respect to $Q$,
and for $\Omega_{\F_2,Q}$, with respect to $\partial \Omega_{\F_2,Q}$ (where we view the latter as a
surface ball on itself, of radius $\approx \ell(Q)$).  That such a common
Corkscrew point exists is proved in \cite[Proposition 6.4]{HM-I}.
Then by \cite[Lemma B.6, Lemma 6.15]{HM-I}, we have that $\P_{\F_2}^{\,\sigma}\hm^{X_Q}$ belongs to $A_{\infty}^{\rm dyadic}(Q),$
with respect to $\sigma$.
 We further observe that $X_0$,
as chosen in Lemma \ref{lemma:Carleson}, is
effectively a Corkscrew point relative to $Q_0$, by the Harnack Chain condition.  Thus,
by a comparison principle argument (see \cite[Corollary 3.69]{HM-I}),
we may change the pole to obtain that $\P_{\F_2}^{\,\sigma}\hm^{X_{0}}$ belongs
to $A_{\infty}^{\rm dyadic}(Q)$.  Note that
by definition, $\P_{\F_2}^{\,\sigma}= \P_{\F}^{\,\sigma}\P_{\F_0}^{\,\sigma}$. Using this fact, and
the definition of $\hm_0$ (see \eqref{eq2.hm0def}),  we  obtain as desired \eqref{eq1.14+}:
$$
\Pf^\sigma \hm_0
=
\P_{\F}^{\,\sigma}\P_{\F_0}^{\,\sigma}\hm^{X_0}
=
\P_{\F_2}^{\,\sigma} \hm^{X_0}
\in A_{\infty}^{\rm dyadic}(Q).
$$

\subsection{Step 3: UR for $E$}

From the previous step, and recalling that  abusing the notation $\Omega$ is really $\Omega_N$, we know that $\pom_N$ is UR with uniform bounds. Then we invoke \cite[Section 2.6, Remark 2.93]{HMU} and conclude that $\pom$ is UR as desired.

\end{document}